\documentclass[11pt, a4paper, oneside]{amsart}  
\usepackage{amsmath, amscd} 

\usepackage{amssymb} 
\usepackage{xcolor} 
\usepackage{hyperref}
\usepackage{enumitem}
\usepackage{pst-node}
\usepackage{tikz-cd} 
\usepackage{graphicx} 
\usepackage{comment}
\usepackage{xcolor}
\usepackage[textsize=small]{todonotes}

\newtheorem{thm}{Theorem}
\newtheorem{prop}{Proposition}[section]
\newtheorem{lemma}[prop]{Lemma}

\theoremstyle{definition}

\newtheorem{remark}[prop]{Remark}

\newcommand{\rr}{\mathbb{R}}
\newcommand{\nn}{\mathbb{N}}
\newcommand{\zz}{\mathbb{Z}}
\newcommand{\qq}{\mathbb{Q}}

\newcommand{\hh}{\mathbb{H}}
\newcommand{\ff}{\mathbb{F}}
\newcommand{\kk}{\mathbb{K}}

\newcommand{\Fp}{ \mathbb{F}_{>0}}

\title[An incomplete real tree with complete segments]{An incomplete real tree \\ with complete segments}
\author{Raphael Appenzeller, Luca De Rosa, Xenia Flamm, Victor Jaeck}
\date{\today}      
\address{Department of Mathematics, ETH Z\"{u}rich, Switzerland}
\email{raphael.appenzeller@math.ethz.ch}

\address{Department of Mathematics, ETH Z\"{u}rich, Switzerland}
\email{luca.derosa@math.ethz.ch}

\address{Department of Mathematics, ETH Z\"{u}rich, Switzerland}
\email{xenia.flamm@math.ethz.ch}

\address{Department of Mathematics, ETH Z\"{u}rich, Switzerland}
\email{victor.jaeck@math.ethz.ch}

\def\subjclassname{\textup{2020} Mathematics Subject Classification}
\expandafter\let\csname subjclassname@1991\endcsname=\subjclassname

\begin{document}

\begin{abstract}
Let $\ff$ be the field of real Puiseux series and $\mathcal{T}_\ff$ the $\qq$--tree defined by Brumfiel in \cite{Brumfiel_TreeNonArchimedeanHyperbolicPlane}.
We show that completing all the segments of $\mathcal{T}_\ff$ does not result in a complete metric space.
\end{abstract}

\maketitle

\section{Introduction}
Real trees, or more generally $\Lambda$--trees, appear in various ways in the study of degenerations of isotopy classes of marked hyperbolic structures on surfaces, see for example \cite{Brumfiel_RSCTeichmuellerSpace, MorganShalen_ValuationsTressDegenerationsHyperbolicStructuresI, MorganShalen_FreeActionsRealTrees}, and more recently \cite{BurgerIozziParreauPozzetti_PositiveCrossratiosBarycentersTrees}.
If $\Lambda$ is a dense subgroup of $(\rr,+)$
and $\mathcal{T}$ is a $\Lambda$--tree, we can define the real tree
\[\mathcal{T}^{\operatorname{sc}} := \bigcup _{s \text{ segment in } \mathcal{T}} \overline{s},\]
where $\overline{s}$ is the metric completion of the segment $s$ in $\mathcal{T}$.
Then the $\Lambda$--tree $\mathcal{T}$ isometrically embeds in $\mathcal{T}^{\operatorname{sc}}$, which is an $\rr$--tree.
The completeness of $\rr$--trees has been addressed in \cite{ChiswellMuellerSchlagePuchta_CompactnessRealTrees}.

Let now $\kk$ be a non-Archimedean valued real closed field with value group $\Lambda < \rr$, and $\mathcal{T}_\kk$ its associated $\Lambda$--tree defined as in \cite{Brumfiel_TreeNonArchimedeanHyperbolicPlane}, see Section \ref{section:prelim}.
We call $\mathcal{T}_\kk^{\operatorname{sc}}$ the \emph{segment completion tree} associated to $\kk$.
We show that the latter is in general not a complete metric space.
\begin{thm}\label{thm:mainthm}
Let $\ff$ be the field of real Puiseux series with $\qq$--valuation, and $\mathcal{T}_\ff$ its associated $\qq$--tree.
Then $\mathcal{T}_\ff^{\operatorname{sc}}$ is not metrically complete. 
\end{thm}

This result as well as the main idea of the proof is perhaps well-known to those working in the field, but the authors are not aware of an explicit example in the literature.
To prove Theorem~\ref{thm:mainthm} we construct a Cauchy sequence in $\mathcal{T}_\ff$ that does not converge in $\mathcal{T}_\ff^{\operatorname{sc}}$.
This sequence is contained in infinitely many different isometric copies of $\qq$ in $\mathcal{T}_\ff$.
Intuitively it can be thought of as a sequence with infinitely many branching points, see Figure \ref{fig:pic}.
A written note of Anne Parreau inspired the explicit Cauchy sequence not converging in $\mathcal{T}_\ff^{\operatorname{sc}}$.
The authors expanded on her example and proved that this sequence does not converge.

Affine $\Lambda$--buildings appear in \cite{BurgerIozziParreauPozzetti_PositiveCrossratiosBarycentersTrees} and \cite{BurgerIozziParreauPozzetti_RSCCharacterVarieties} in the study of degenerations of (higher rank) Teichm\"uller representations.
Their metric completions are $\textrm{CAT(0)}$  for a suitable metric, see \cite[Proposition 10]{BurgerIozziParreauPozzetti_RSCCharacterVarieties}.
A natural question is whether such completions are obtained by completing every apartment. In \cite{BruhatTits} Bruhat and Tits construct an affine building, called the \emph{Bruhat--Tits building}, given a reductive algebraic group and a valued field.
For $\operatorname{SL}_2(\kk)$, we expect the Bruhat--Tits building to be isometric to the segment completion tree associated to $\kk$.
Bruhat--Tits give a general characterisation of completeness of Bruhat--Tits buildings \cite[Theorem 7.5.3]{BruhatTits}.
In the case $\operatorname{SL}_2(\kk)$, their result implies that the Bruhat--Tits building is complete if and only if the field $\kk$ is \emph{spherically complete}, i.e.\ every decreasing sequence of balls in $\kk$ has non-empty intersection.
The field of real Puiseux series is not spherically complete, hence our result is in line with Bruhat--Tits' characterization assuming that the segment completion tree associated to $\kk$ is the affine building associated to $\operatorname{SL}_2(\kk)$.
The advantage of our approach is that the trees in question are more easily defined and we do not need the machinery built up in \cite{BruhatTits}.

After introducing the necessary background in Section \ref{section:prelim}, we prove Theorem \ref{thm:mainthm} in Section \ref{section:proofthm}.
We thank Anne Parreau for her support and for pointing us to the right chapter in \cite{BruhatTits}.
We are thankful to Marc Burger, Anne Parreau and Beatrice Pozzetti for constructive discussions and feedback.

\section{Preliminaries}\label{section:prelim}

An ordered field $\kk$ is called \emph{real closed} if every positive element is a square and every polynomial of odd degree has a root.
It is \emph{non-Archimedean} if there is an element that is larger than any $n \in \nn \subseteq \kk$.
For an ordered abelian group~$\Lambda$, a $\Lambda$--\emph{valuation} of $\kk$ is a map $v \colon \kk \to \Lambda \cup \{\infty\}$, that satisfies $v(a)=\infty$ if and only if $a=0$, $v(ab) = v(a)+v(b)$ and $v(a+b) \geq \min\{ v(a), v(b) \}$ with equality if $v(a) \neq v(b)$, for all $a, b\in \kk$.
It is \emph{order compatible} if for $a \leq b$ one has $\nu (b) \leq \nu (a)$.
For the remainder of this text we work with the following valued real closed field $\ff$, where $\Lambda = \qq$.
The field of \emph{real Puiseux series} is the set
 \[
 \ff := \left\{  \sum_{k=-\infty}^{k_0} c_k X^{\frac{k}{m}} \, \Bigg| \, k_0, m \in \mathbb{Z}, \, m > 0 , \, c_k \in \rr, \, c_{k_0}\neq 0\right\},
 \]
with the order such that $X$ is larger than any real number.
The real Puiseux series form a non-Archimedean, real closed extension of the ordered field of rational functions in one variable $\rr (X)$ with the compatible order \cite[Example 1.3.6.b)]{BochnakCosteRoy_RealAlgebraicGeometry}.
The \emph{logarithm}
\[\log \colon  \ff_{>0} \to \qq, \quad \sum_{k=-\infty}^{k_0} c_k X^{\frac{k}{m}}  \mapsto \tfrac{k_0}{m} \]
is order preserving and $-\log | \cdot |$ is an order compatible  $\qq$--valuation of $\ff$ by setting $\log(0):=-\infty$. In particular we have for all $a,b \in \ff_{>0}$, $\log(ab) = \log(a) + \log(b)$ and $\log(a+b) \leq \max \{\log(a), \log(b)\}$ with equality if $\log(a) \neq \log(b)$.
Following the construction in \cite{Brumfiel_TreeNonArchimedeanHyperbolicPlane}, we define the \emph{non-Archimedean hyperbolic plane over $\mathbb{F}$} as the set 
\[\mathbb{H}_\mathbb{F} := \{x+iy \mid x,y \in \mathbb{F}, y >0 \} \subseteq \ff[i], \]
where $i$ is such that $i^2 =-1$. Mimicking the real case, there is a pseudo-distance $d$ defined on $\hh_\ff$ as follows.
Given $z=x+iy$ and $z'=x'+iy'$ in $\mathbb{H}_\mathbb{F}$, consider the unique \emph{$\ff$--line} passing through them.
That is, either the vertical ray going through these two points if $x=x'$, or the half-circle through $z$ and $z'$ that meets $\{y=0\}$ in a right angle.
Denote the endpoints of the $\ff$--line by $w$ and $w'$ such that $w,z,z',w'$ appear in this order on the $\ff$--line, see Figure ~\ref{fig:crossratio}.
\begin{figure}[t]
\centering
\includegraphics[scale=0.7]{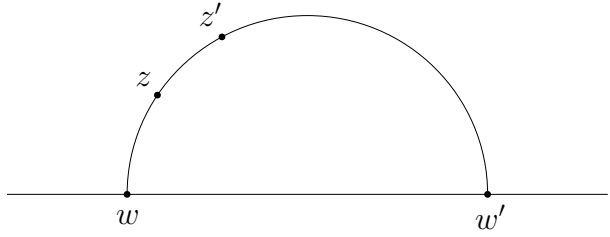}
\caption{Two points $z$ and $z'$ in $\hh_\ff$ determine a quadruple $(w,z,z',w')$, to which a cross-ratio $\textrm{CR}(w,z,z',w') \in \ff_{\geq1}$ can be associated.}
\label{fig:crossratio}
\end{figure}
Since $\ff$ is real closed, we can define the cross-ratio $\mathrm{CR}$ on $\hh_\ff$ analogously as for the real hyperbolic plane.
Brumfiel \cite[Equation (18)]{Brumfiel_TreeNonArchimedeanHyperbolicPlane} showed that
\[d(z,z') := \log \textrm{CR}(w,z,z',w') \in \rr_{\geq 0}\]
is a pseudo-distance on $\hh_\ff$. Note that two points can have distance zero because $\log$ is not injective. By the properties of $\log$, we obtain the following description of the pseudo-distance, which we use from now on, see \cite[Equation (19)]{Brumfiel_TreeNonArchimedeanHyperbolicPlane}.

\begin{remark} \label{rem:comp_dist}
Let $z= x + iy$ and $z' = x' + i y'$ be points in $\hh_\ff$, then we have
\begin{align*}
d(z, z') &= \log \left( \frac{{(x - x')}^2 + y^2 + {y'}^2}{y y'} \right) \\
&= \max \left\{ \log \left(\frac{{(x - x')}^2}{y y'} \right) ,\, \log \left( \frac{y}{y'} \right),\, \log \left( \frac{y'}{y} \right) \right\}.
\end{align*}
\end{remark}
Consider the equivalence relation on $\hh_\ff$ which identifies $z$ with $z'$ if and only if $d(z,z') = 0$ and denote by $\pi$ the projection
\[
\pi: \hh_\ff \to \hh_\ff \big/ \{d=0\} =: \mathcal{T}_\ff.
\]
In \cite[Theorem (28)]{Brumfiel_TreeNonArchimedeanHyperbolicPlane}, Brumfiel showed that $\mathcal{T}_\ff$ is a $\Lambda$--tree, with $\Lambda = \qq$.
We now define $\Lambda$--trees following \cite{Chiswell_LambdaTrees}.
Let $\Lambda$ be an ordered abelian group. A set $X$ together with a function $d \colon X \times X \to \Lambda$ is a $\Lambda$--\emph{metric space} if $d$ is positive definite, symmetric and satisfies the triangle inequality.
The group $\Lambda$ is itself a $\Lambda$--metric space by defining $d(a,b) := |b-a| := \max\{a-b, b-a\}$.
A \emph{segment} $s$ is the image of a $\Lambda$--isometric embedding $ \varphi \colon [a,b]_\Lambda = \{{t \in \Lambda \colon } a \leq t \leq b\} \to X$ for some $a \leq b \in \Lambda$, and the set of \emph{endpoints of $s$} is $\{\varphi(a),\varphi(b)\}$.
We call the image of an isometric embedding of $\Lambda$ a \emph{$\Lambda$--geodesic}.
A $\Lambda$--\emph{tree} is a $\Lambda$--metric space satisfying the following three properties:
\begin{enumerate}[label=(\alph*)]
\item \label{defLambdaTree:axiom1}
For all $x,y \in X$ there is a segment $s$ with endpoints $x,y$.
\item \label{defLambdaTree:axiom2}
For all segments $s,s'$ whose intersection $s \cap s' = \{x\}$ consists of one common endpoint $x$ of both segments, the union $s \cup s'$ is a segment.
\item \label{defLambdaTree:axiom3} 
For all segments $s,s'$ with a common endpoint $x$, the intersection $s\cap s'$ is a segment with $x$ as one of its endpoints.
\end{enumerate}
By \cite[Lemma II.1.1]{Chiswell_LambdaTrees}, segments in a $\Lambda$--tree are unique.
We write $[x,y]$ for the unique segment with endpoints $x,y \in X$.
We use the following statement later.

\begin{lemma}\label{lem:tree}
Let $X$ be a $\Lambda$--tree and $x,y,z \in X$. Then $[y,z] \subseteq [x,y] \cup [x,z]$.
\end{lemma}
\begin{proof}
By axiom \ref{defLambdaTree:axiom3} of $\Lambda$--trees there is a point $r \in X$ with $[x,r] = [x,y]\cap [x,z]$.
Since then $[r,y]\cap [r,z] = \{r\}$, we have by axiom \ref{defLambdaTree:axiom2} (and uniqueness of segments) that $[y,z] = [y,r]\cup [r,z] \subseteq [x,y]\cup [x,z]$.
\end{proof}

\section{Proof of Theorem \ref{thm:mainthm}}\label{section:proofthm}

This section is concerned with the proof of Theorem \ref{thm:mainthm}, which states that the segment completion tree $\mathcal{T}_\ff^{\operatorname{sc}}$ associated to the field of real Puiseux series $\ff$, is not metrically complete.
We construct an explicit Cauchy sequence $(\pi(p_n))_{n \in \nn}$ in $\mathcal{T}_\ff$, whose isometric embedding does not converge in $\mathcal{B}\mathcal{T}_\ff$.

Let us first define a sequence of rational exponents $t_n$ to define points $p_n := a_n + i b_n \in \hh_\ff$ whose projections $\pi(p_n) \in \mathcal{T}_\ff$ will form the Cauchy sequence that does not converge.
Let $t_0:= 0 \in \qq$, $a_0:=0$, $b_0:=1 \in \ff$ and
for $n\geq 1$ define
\begin{align*}
   t_n &:= \sum_{k=1}^n -\frac{1}{2^k} = -\frac{2^n-1}{2^n} = -1+\frac{1}{2^n} \in \qq,\\ 
a_n &:= \sum_{k=1}^n X^{t_k} = X^{t_1} + X^{t_2} + \cdots + X^{t_n} \in \ff_{>0}, \\
b_n &:= X^{t_n} \in \ff_{>0}.
\end{align*}
We note that $t_n$ is a monotonically decreasing sequence converging to $-1$, $a_n$ is monotonically increasing and $b_n$ is monotonically decreasing.
Figure \ref{fig:sequenceinupperhalfplane} contains a schematic picture of the sequence $(p_n=a_n+ib_n)_{n \in \nn}$ in $\hh_\ff$.

\begin{figure}[h]
\centering
\includegraphics[scale=1]{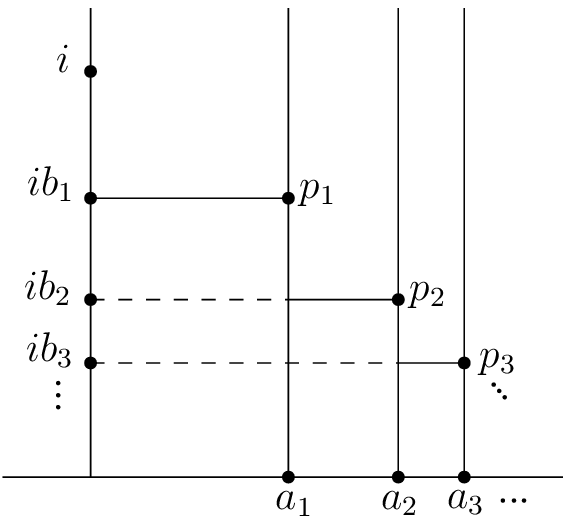}
\caption{The sequence $(p_n)_{n \in \nn}$ in the upper half plane $\hh_\ff$.} 
\label{fig:sequenceinupperhalfplane}
\end{figure}

\begin{lemma} \label{lem:cauchy}
 The sequence $(p_n)_{n\in \mathbb{N}}$ is a Cauchy sequence in $\hh_\ff$. 
\end{lemma}
\begin{proof}
Let $n\leq m$.
Then $t_m \leq t_n$, and, using the distance formula from Remark \ref{rem:comp_dist}, we obtain
\begin{align*}
    d(p_n, p_{m}) &= \log \left( \frac{(a_m - a_n)^2 + b_n^2 + b_m^2 }{b_nb_m} \right)\\
    &  =\log \left( \left(\sum_{k=n+1}^m X^{t_k}\right)^2 + \left(X^{t_n}\right)^2 + \left(X^{t_m}\right)^2  \right) - \log \left( X^{t_n}X^{t_m} \right) \\
    & = 2 t_{n} - (t_n + t_m) = t_n - t_m \to 0 \quad \text{ (as } n, m \to \infty )
\end{align*}
and thus $(p_n)_{n\in \nn}$ is a Cauchy sequence.
\end{proof}

The following lemma is a general fact about the distance function in $\hh_\ff$. 
\begin{lemma} \label{lem:verticalraysidentified}
For every $x, x'$ in $\ff$, there exists $y \in \ff_{>0}$ such that 
\[d\left(x+iy, x'+iy\right) = 0.\]
Furthermore, for all $y \in \ff_{>0}$ with $d(x+iy, x'+iy) = 0$ and all $t\geq 0$, it holds that $d(x+i(y+t), x'+i(y+t)) = 0$.
\end{lemma}
\begin{proof}
Define $y:= X^{\log|x-x'|}>0$ so that, using Remark \ref{rem:comp_dist}, we obtain
\[d(x+iy, x'+iy) = \max \left\{ \log \left(\left(x-x'\right)^2\right) - \log X^{2\log|x-x'|},\, 0\right\} = 0.\]
The second equality follows from $\log(y+t) \geq \log(y)$ and that hence the expression $\log ((x-x')^2) - \log (X^{2\log|x-x'|}+t)$ is negative.
\end{proof}

Intuitively Lemma  \ref{lem:verticalraysidentified} tells us that two vertical $\ff$--lines in $\hh_\ff$ are identified from some point on in $\mathcal{T}_\ff$, forming an infinite tripod.
In Lemma \ref{lem:sequenceinquotient}, we refine this statement for the sequence of vertical $\ff$--lines $\ell_n := a_n + i \Fp$.
The situation is illustrated in Figure \ref{fig:pic}.

\begin{figure}[h]
\centering
\includegraphics[scale=1]{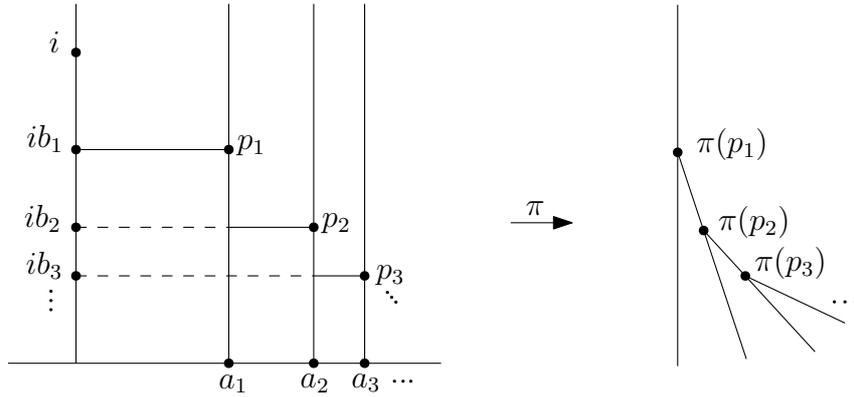}
\caption{The points $p_n \in \ell_n \subseteq \hh_\ff$ correspond to branching points $\pi(p_n)$ of $\pi(\ell_{n-1})$ and $\pi(\ell_n)$ in $\mathcal{T}_\ff$.} 
\label{fig:pic}
\end{figure}

\begin{lemma}\label{lem:sequenceinquotient}
Let $n\in \nn, \, b \in \Fp$. Consider the vertical $\ff$--lines $\ell_n$, $\ell_{n+1}$. 
\begin{enumerate}[label=(\roman*)]
\item \label{lem:sequenceinquotient:item:branchingpoints}
If $\log (b)  \geq \log( b_{n+1} ) = t_{n+1}$, then $d(a_n + ib ,a_{n+1}+ib)=0$.

\item If $\log(b) < \log(b_{n+1})={t_{n+1}} $, then all points in $\ell_{n+1}$ have non-zero distance to $a_n+ib$.
\end{enumerate}
\end{lemma}
\begin{proof}
\noindent
\begin{itemize}[leftmargin=0.65cm]
  \item[(i)] If $\log(b)=\log(b_{n+1})$, we use $\log(a_{n+1}-a_n) = \log(b_{n+1})$ to see that
    \[
        \quad \quad d(a_n + ib , a_{n+1} + ib) = \max \left\{ \log \left( \frac{(a_{n+1}-a_n)^2}{b_{n+1}^2} \right) ,0\right\}=0.
    \]
    If $\log(b) > \log(b_{n+1})$, then also $b > b_{n+1}$ and the statement follows from the second part of Lemma \ref{lem:verticalraysidentified}.
    \item[(ii)] Let $a_{n+1}+ib'$ be a point in $\ell_{n+1}$ for some $b' \in \Fp$.
First, assume $b'  < b_{n+1} = X^{t_{n+1}}$, in which case $\log (bb') < 2t_{n+1}$ holds.
Then
    \begin{align*}
        d(a_n+ib,a_{n+1}+ib') &= \log \left(\frac{\left(a_{n+1}-a_n\right)^2+b^2+{b'}^2}{b b'}\right) \\
        &= \log \left( X^{2t_{n+1}} + b^2 + {b'}^2\right) - \log (b b') \\
        &= \log \left( X^{2t_{n+1}} \right) - \log (b b') >0.
    \end{align*}
    Second, assume $b' \geq b_{n+1} = X^{t_{n+1}}$.
    We use \ref{lem:sequenceinquotient:item:branchingpoints} to see that $d({a_n+ib'}, \allowbreak a_{n+1}+ib') = 0$.
    Then, using $d(a_n+ib, a_n + ib') = \log( b' ) - \log(b) > 0$, we conclude
    \begin{align*}
        \quad\quad d\left(a_n+ib, a_{n+1} + ib'\right) &= d\left(a_n+ib, a_{n+1}+ib'\right) + d\left(a_n+ ib', a_{n+1}+ib'\right) \\
        & \geq d\left(a_n+ib,a_n+ib'\right) > 0.
    \end{align*}
\end{itemize}
\end{proof}

Lemma \ref{lem:sequenceinquotient} implies that for all $n \in \nn$ the $\qq$--geodesics $\pi(\ell_n)$ and $\pi(\ell_{n+1})$ in $\mathcal{T}_\ff$ branch off at the point $\pi(p_{n+1})$.
By Lemma \ref{lem:cauchy}, $(p_n)_{n\in \nn}$ and hence also $(\pi(p_n))_{n \in \nn}$ are Cauchy sequences.
The next proposition shows that $(\pi(p_n))_{n \in \nn}$ does not converge in the $\mathbb{R}$--tree $\mathcal{T}_\ff^{\operatorname{sc}}$.
Intuitively, the sequence $(\pi(p_n))_{n\in \nn}$ does not stay in $\pi(\ell_m)$ for any $m\in \nn$, and hence the limit point is not contained in the completion of any of the $\qq$--geodesics $\pi(\ell_m)$.
In fact, the following proposition shows that it is not contained in the completion of any $\qq$--geodesic in $\mathcal{T}_\ff$. The sequence $(\pi(p_n))_{n\in \nn} \subseteq \mathcal{T}_\ff$ is a Cauchy sequence by Lemma \ref{lem:cauchy} and thus Theorem \ref{thm:mainthm} follows from the following proposition.

\begin{prop}\label{prop:main}
The sequence $(\pi(p_n))_{n\in \mathbb{N}}$ does not converge in $\mathcal{T}_\ff^{\operatorname{sc}}$.
\end{prop}

\begin{proof}

We assume by contradiction that $\pi(p_n)$ converges to some $\overline{p} \in \mathcal{T}_\ff^{\operatorname{sc}}$.
Since $\mathcal{T}_\ff^{\operatorname{sc}}$ is by definition the union of the completions of the segments of $\mathcal{T}_\ff$, we can find a segment $s$ in $\mathcal{T}_\mathbb{F}$, such that $\overline{p}$ lies in its completion $\overline{s}$.

\textbf{Step 1.} Our first goal is to show that we may assume that $s$ is contained in the image of a vertical $\ff$--line $\ell$ of $\hh_\ff$.
The segment $s$ has two endpoints in the tree $\mathcal{T}_{\ff}$.
Choose preimages $z_1=x_1+iy_1$ and $z_2=x_2+iy_2 \in \hh_{\ff}$ so that $\pi(z_1)$ and $\pi(z_2)$ are the endpoints of $s$. 
By Lemma \ref{lem:verticalraysidentified}, there is a large enough $y \in \ff_{>0}$ with $d\left(x_1+iy , x_2 + iy\right) = 0$, hence $\pi(x_1+iy)={\pi(x_2+iy)} \in \mathcal{T}_\ff$.
Lemma \ref{lem:tree} implies that $s = \left[\pi(z_1),\pi(z_2)\right] \subseteq \left[\pi(x_1+iy),\pi(z_1)\right]\cup \left[\pi(x_2+iy), \pi(z_2)\right]$.
We conclude that at least one of the completions of the segments $[\pi(z_1),\allowbreak {\pi(x_1+iy)}]$ or $[\pi(z_2),\pi(x_2+iy)]$ has to contain $\overline{p}$.
Without loss of generality $\overline{p}$ lies in the completion of $[\pi(z_1),\pi(x_1+iy)]$.
Set $a := x_1 \in \ff$ so that the vertical $\ff$--line $ \ell := a+i\ff_{>0}$ is such that $\overline{p}$ lies in the completion of $\pi(\ell)$.

\textbf{Step 2.} 
We define points 
$ q_n = a + iX^{t_n} \in \ell \subseteq \hh_\ff$ and claim that $d (p_n , q_n ) = 0$ for all $n \in \nn$, see Figure \ref{fig:p_to_q}.
To show this, fix $n$ and let $N\in \nn$ be large enough so that $d(\pi(p_N), \overline{p} ) < d(\pi(p_n) , \overline{p})$.
Consider the vertical $\ff$--line $\ell_N = a_N + i \ff_{>0} \subseteq~\hh_\ff$ that contains both $p_N$ and the point $p_n'=a_N+iX^{t_n}$.
By Lemma \ref{lem:sequenceinquotient} \ref{lem:sequenceinquotient:item:branchingpoints}, $d(p_n, p_n') = 0$ and hence $\pi (p_n) = \pi( p'_n) \in \pi(\ell_N)$.
By Lemma \ref{lem:verticalraysidentified} we can find points $p_{-1} = a_N+ib \in \ell_N$ and $q_{-1} = a+ib\in \ell$ with $d(p_{-1}, q_{-1}) = 0$.
We distinguish two cases, and show that in fact the second case cannot occur.

\vspace{2pt}
\begin{figure}[t]
\includegraphics[scale=1]{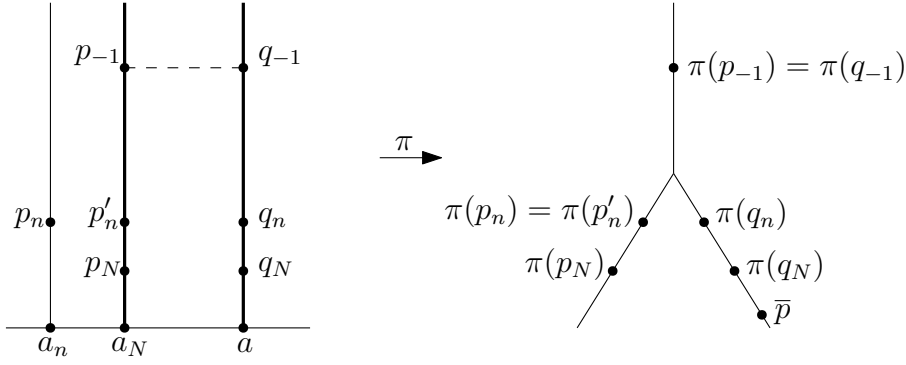}
\caption{The setup to prove $\pi(p_n) = \pi(q_n)$ in Step 2. The tripod on the right is the image under $\pi$ of the two vertical lines in bold on the left.}
\label{fig:p_to_q}
\end{figure}

Case 1: 
 $\pi(p_n) \in \pi(\ell)$.
 Then there exists $y \in \ff_{>0}$ such that $\pi(a+iy) = \pi(p_n)  \in \pi(\ell)$.
 But this means nothing else than $ d( p_n , a+iy)=0$.
 We use $p_n = a_n + i X ^{t_n}$ and Remark \ref{rem:comp_dist} to get 
\[
0 = d( p_n , a+iy)  = \max \left\{ \log \left(\frac{(a-a_n)^2}{X^{t_n}y} \right), {t_n} - \log(y) , \log (y) - {t_n} \right\}.
\]
Hence $\log y = {t_n}$ so that we conclude
\[
d(p_n, q_n) =\max \left\{ \log \left( \frac{(a-a_n)^2}{X^{2t_n} }\right) , 0 \right\}=  d(p_n, a+ iy) =   0.
\]

Case 2: $\pi(p_n)\in \pi(\ell_N) \setminus \pi(\ell)$.
The goal is to show that in fact this case cannot occur.
Observe that $\pi(p_N) \in \pi(\ell_N) \setminus \pi(\ell)$, since otherwise $\pi(p_N)= \pi(q_N)$ as in Case 1, and the second part of Lemma \ref{lem:verticalraysidentified} would then imply that $\pi(p_n)=\pi(q_n) \in \pi (\ell)$.
By Lemma \ref{lem:tree} we have for every $r \in \pi(\ell)$, $[\pi(p_N),\pi(p_{-1})] \subseteq [r,\pi(p_N)]\cup [r,\pi(p_{-1})]$.
Since $\pi(p_n) \in [\pi(p_N),\pi(p_{-1})]\setminus \pi(\ell)$ and $[r,\pi(p_{-1})] \subseteq \pi(\ell)$, we have $\pi(p_n) \in [\pi(p_N), r]$ and hence 
\[
d(\pi(p_N),\pi(p_n)) + d(\pi(p_n),r) =  d(\pi(p_N),r).
\]
By density of $\qq$ in $\rr$, choose $r \in \pi(\ell)$ close to $\overline{p}$ in the sense that $d(r,\overline{p}) < d (p_n, p_N)$.
Without loss of generality we may assume that $ d(\pi(p_N),r) < d(\pi(p_N),\overline{p})$.
We then have
\begin{align*}
d(\pi(p_n),\overline{p}) &\leq d(\pi(p_n),r) + d(r,\overline{p}) < d(\pi(p_n), r) + d(\pi(p_n),\pi(p_N)) \\ 
    &= d(\pi(p_N),r) < d(\pi(p_N),\overline{p}) < d(\pi(p_n), \overline{p}),
\end{align*}
where the last inequality comes from the choice of $N$.
This is a contradiction.

\textbf{Step 3.} Next we would like to get a bound on $\log \left|a-a_n\right|$. 
We use Remark~\ref{rem:comp_dist} to obtain
\begin{align*}
    0 & = d(p_n, q_n) = \max \left\{ \log \left(\frac{(a-a_n)^2}{X^{2t_n}}\right) , 0 \right\} \\
    & = \max \left\{ 2\left( \log \left|a-a_n\right| -t_n \right) ,0 \right\} ,
\end{align*}
which implies $\log | a-a_n| \leq t_n$.

\textbf{Step 4.}
Recall that $a_n = \sum_{j=1}^n X^{t_j}$ and $t_j=(1-2^j)/2^j$ is monotonically decreasing in $j$.
By definition of $\ff$, there exist $m\in \mathbb{N}$, $k_0\in \mathbb{Z}$ and $c_k \in \rr$ such that $a \in \ff$ can be written as
\[
a = \sum_{k=-\infty }^{k_0} c_k X^{\frac k m}.
\]
Let us also write for every $n \in \nn$
\[
a - a_n = \sum_{k=-\infty}^{k_{0,n}} d_{k,n} X ^{\frac{k}{m_n}},
\]
for $m_n \in \nn$, $k_{0,n} \in \zz$ and $d_{k,n} \in \rr$, such that $k_{0,n} := \max \{ k  \in \zz \colon d_{k, n} \neq 0 \}$. 
From Step 3 we obtain that
\[
\frac{k_{0,n}}{m_n}= \log\left|a - a_n\right| \leq t_{n}.
\]
Thus for every $n\in \nn$ we have 
\[
a - a_n  = \sum_{k=-\infty }^{k_0} c_k X^{\frac k m} - \sum_{j = 1}^n X^{t_j} =  \sum_{k = -\infty}^{k_{0,n}} d_{k, n} X^{\frac{k}{m_n}},
\]
where $k_{0,n}/m_n$ is less or equal than $t_n$.
Since the sequence $(t_n)_{n \in \nn}$ is monotonically decreasing, it means that the series expansion for $a$ contains the terms $X^{t_j}$ for all $j \in \{1,\ldots, n-1 \}$ which cancel the terms $X^{t_j}$ in $a_n$.
In other words $m t_j\in \zz$, which is equivalent to $m/2^j \in \zz$ for all $j \in \{1,\ldots, n-1\}$ by definition of $t_j$.
As this holds for every $n\in \nn$, we obtain a contradiction, since there is no $m\in \zz$ such that $m/2^j\in \zz$ for all $j \in \nn$.
\end{proof}


\bibliographystyle{amsalpha}
\bibliography{main}

\end{document}